[1994/12/01]
\documentclass{amsart}
\usepackage[dvipdfmx]{color}
\usepackage[dvipdfmx]{graphicx}

\title[Conjugation-invariant norms on ${[ B_\infty,B_\infty ]}$]{Conjugation-invariant norms on the commutator subgroup of the infinite braid group}
\author{Mitsuaki Kimura}
\address{Graduate School of Mathematical Sciences, the University of Tokyo
3-8-1 Komaba, Meguro-ku, Tokyo, 153-8914, Japan}
\email{mkimura@ms.u-tokyo.ac.jp}
\date{}

\usepackage{amsmath,amsthm}

\chardef\bslash=`\\ 

\newtheorem{thm}{Theorem}[section]
\newtheorem{cor}[thm]{Corollary}
\newtheorem{lem}[thm]{Lemma}
\newtheorem{prob}[thm]{Problem}
\newtheorem{prop}[thm]{Proposition}

\theoremstyle{definition}
\newtheorem{defn}[thm]{Definition}

\theoremstyle{remark}

\DeclareMathOperator{\cl}{cl}

\newcommand{\eval}[2][\right]{\relax
  \ifx#1\right\relax \left.\fi#2#1\rvert}

\begin{document}

\maketitle

\begin{abstract}
In this paper, we give a proof of the result of Brandenbursky and K\c{e}dra which says that the commutator subgroup of the infinite braid group admits stably unbounded norms. Moreover, we observe the norms which we constructed are equivalent to the biinvariant word norm studied by Brandenbursky and K\c{e}dra.

\end{abstract}

\section{Introduction}

In \cite{bip}, Burago, Ivanov, and Polterovich 
introduced the notion of conjugation-invariant norms and asked several problems. One of them is as follows:

\begin{prob}[\cite{bip}]\label{bip}
Does there exists a perfect group $G$ which satisfies the following conditions?
\begin{enumerate}
\item The commutator length
is stably bounded on $[G,G]$.
\item $G$ admits a stably unbounded norm.
\end{enumerate}
\end{prob}

\noindent
For the definitions on conjugation-invariant norms, see \cite{bip}. 
It is known that such groups exist. 
Brandenbursky and K\c{e}dra \cite{bk} proved the following theorem:

\begin{thm}[\cite{bk}] \label{mainthm}
The commutator subgroup of the infinite braid group $[B_\infty,B_\infty]$ admits a stably unbounded norm.
\end{thm}

Kawasaki \cite{kaw} also showed that
the commutator subgroup of ${\rm Symp}_c(\mathbb{R}^{2n})_0$ is such a group, 
where ${\rm Symp}_c(\mathbb{R}^{2n})_0$ is the group of symplectomorphisms with compact support isotopic to the identity of the standard symplectic space.
In this paper, we give a proof of Theorem\ref{mainthm} by using the idea of Kawasaki in \cite{kaw}. 
Kawasaki introduced $\nu$-quasimorphisms (or relative quasimorphisms)
and the $(\nu,p,q)$-commutator length $\cl_{\nu,p,q}$.
He proved that the existence of non-trivial $\nu$-quasimorphism implies stably unboundedness of $\cl_{\nu,p,q}$ (see Proposition \ref{kaw}). 

To give our proof of Theorem\ref{mainthm}, we construct stably unbounded norms $\cl_{\nu_n,p,q}$ on $[B_\infty, B_\infty]$ by observing that the signature of braids $\sigma : B_{\infty} \to \mathbb{R}$ is a $\nu_n$-quasimorphism in \S 2.1. 
In \S 2.2,  we study the property of the norms $\cl_{\nu_n,p,q}$ and prove the following theorem:

\begin{thm} \label{equiv2}
For any integer $n\geq 2$ and any real numbers $p \geq 1$ and $q \geq 1$, the norm $\cl_{\nu_n,p,q}$ is equivalent to the biinvariant word norm $\|\cdot\|$ (whose stably unboundedness is observed by Brandenbursky and K\c{e}dra).
\end{thm}

\subsection*{Acknowledgement}
The author would like to thank Professor Takashi Tsuboi for his guidance and helpful advice. 
He also thanks to Morimichi Kawasaki for his useful advice.

\section{Proofs of main results}

\subsection{The construction of stably unbounded norms}

First, we explain the idea of Kawasaki. The definitions of $(\nu,p,q)$-commutator length and $\nu$-quasimorphisms are following.

\begin{defn}[\cite{kaw}]

Let $G$ be a group with a conjugation-invariant norm $\nu$ and $p,q \in \mathbb{R}_{>0}$ .
We define the \emph{$(\nu,p,q)$-commutator subgroup}
$[G,G]_{\nu,p,q}$ to be a subgroup of $G$ generated by the elements
$[f,g] \in G$ such that 
$\nu(f)\leq p$, $\nu(g)\leq q$.

We define \emph{ $(\nu,p,q)$-commutator length} $\cl_{\nu,p,q}:[G,G]_{\nu,p,q}\to \mathbb{R}_{\geq0}$ by
\[\cl_{\nu,p,q}(h)=
\min \left\{ k \; \middle| \; \begin{array}{ll}
\exists  f_i,\exists g_i\in G, \; (i=1,\ldots,k)\\
\nu(f_i)\leq p,\nu(g_i)\leq q
\end{array} 
, \; h=[f_1,g_1]\cdots[f_k,g_k]
  \right\}. \]
\end{defn}

\noindent
We note that $\cl_{\nu,p,q}$ is a conjugation-invariant norm.

\begin{defn}[\cite{kaw}]
Let $G$ be a group with a conjugation-invariant norm $\nu$. A function $\phi : G \to \mathbb{R}$ is called a \emph{$\nu$-quasimorphism} (or \emph{quasimorphism relative to $\nu$}) if there exists a constant $C>0$ such that
for every $f, g \in G$,
\[|\phi(fg)-\phi(f)-\phi(g)|\leq C \min\{\nu(f),\nu(g)\}.\]
\end{defn}

The concept of $\nu$-quasimorphisms appeared earlier in the paper of Entov and Polterovich  (\cite{ep}, Theorem 7.1.), which is called ``controlled quasi-additivity'' in \cite{ep}. 

Kawasaki proved the following proposition. 

\begin{prop}[\cite{kaw}, Proposition 3.4.]\label{kaw}
Let $\phi$ be a $\nu$-quasimorphism on $G$ and $p,q \in \mathbb{R}_{>0}$. 
If there exist an element $h \in [G,G]_{\nu,p,q}$ such that $\lim_{n\to \infty} \frac{\phi(h^n)}{n}>0$, then  $\cl_{\nu,p,q}$ is stably unbounded on $[G,G]_{\nu,p,q}$.
\end{prop}

Next, we give an useful sufficient condition to prove that a function is a 
$q_K$-quasimorphism, where $q_K$ is a conjugation-invariant norm on a group $G$ which is normally generated by a subset $K\subset G$ defined by
\[q_K(f)= \min \{ l \mid \exists k_i\in K, \exists g_i \in G, \; f=k_1^{g_1}\cdots k_l^{g_l} \}, \] 
where $x^y$ denotes the conjugation $xyx^{-1}$.

\begin{lem}[\cite{ep}, Lemma 7.3.]\label{ep} 
Let $G$ be a group which is normally generated by a subset $K\subset G$, and
$\phi:G \to \mathbb{R}$ a function on $G$. If there exists a constant $C>0$ such that the inequality
$|\phi(gh)-\phi(g)-\phi(h)|\leq C$ holds 
for any $g\in G$ and any $h\in G$ with $q_K(h)=1$, then $\phi$ is a $q_K$-quasimorphism.
\end{lem}

Finally, we construct the stably unbounded norms on $[B_\infty,B_\infty]$ by proving the signature of braids is a $\nu_n$-quasimorphism, where $\nu_n=q_{B_n}$.

We denote the $n$-braid group by $B_n$ and let 
$\sigma_1 , \ldots, \sigma_{n-1}$ be the standard Artin generators. 
For a braid $\alpha \in B_n$, we denote the closure of $\alpha$ by  $\widehat{\alpha}$. 
For a braid $\alpha \in B_n$, $\sigma(\alpha)$ is defined to be the signature $\sigma(\widehat{\alpha})$ of the link $\widehat{\alpha}$. 
We consider the standard inclusion $\iota_n : B_n \to B_{n+1}$, $\sigma_i \mapsto \sigma_i$
(i.e. $\iota_n$ is ``adding a trivial string'') and obtain the following sequence : $B_1 \subset B_2 \subset \cdots \subset B_n \subset \cdots$. 
We define the infinite braid group $B_{\infty}$ to be $\bigcup_{n=1}^{\infty} B_n$.
Since the signature of braids $\sigma : B_n \to \mathbb{R}$ and the inclusion $\iota_n : B_n \to B_{n+1}$ are compatible, 
i.e. $\sigma (\iota_n(\alpha)) = \sigma(\alpha)$ for $\alpha \in B_n$, $\sigma : B_\infty \to \mathbb{R}$ is well-defined. 

\begin{thm}\label{rel}
The signature of braids $\sigma : B_\infty \to \mathbb{R}$ is a $\nu_n$-quasimorphism.
\end{thm}
\begin{proof}
By Lemma \ref{ep}, 
it is sufficient to prove that $|\sigma(\alpha\beta)-\sigma(\alpha)-\sigma(\beta)|\leq n$ 
for $\alpha ,\beta \in B_\infty$, $\nu_n(\beta)=1$.
The assumption $\nu_n(\beta)=1$ implies the existence of a braid $\gamma \in B_\infty$ such that $\beta^\gamma \in B_n \subset B_\infty$.
Let $m$ be a natural number such that $\alpha^\gamma \in B_m$ ($m>n$). 
Since the signature of braids is conjugation-invariant, 
\[|\sigma(\alpha\beta)-\sigma(\alpha)-\sigma(\beta)|
=|\sigma(\alpha^\gamma \beta^\gamma )-\sigma(\alpha^\gamma )-\sigma(\beta^\gamma)|.\]

We obtain the link $\widehat{\alpha^\gamma \beta^\gamma}$ from 
$\widehat{\alpha^\gamma} \sqcup \widehat{\beta^\gamma}$ by taking saddle moves $m$ times. Note that the link $\widehat{\beta^\gamma}$ has $m-n$ unknot components since $\beta^\gamma$ has trivial strings after ($n+1$)-th one (Figure \ref{relative-quasi-morphism}). 
It is known that the signature changes at most $\pm$1 by one saddle move (see \cite{mur} for example). Since the signature does not changed by taking connected sum to an unknot, the signature changes at most $n$ by the $m$ times saddle moves. Hence 
\[|\sigma(\alpha^\gamma \beta^\gamma )-\sigma(\alpha^\gamma )-\sigma(\beta^\gamma)|=|\sigma(\widehat{\alpha^\gamma \beta^\gamma})-\sigma(\widehat{\alpha^\gamma} \sqcup \widehat{\beta^\gamma})|\leq n.\]

\end{proof}

\begin{figure}[h]
 \begin{center}
  \includegraphics[width=60mm]{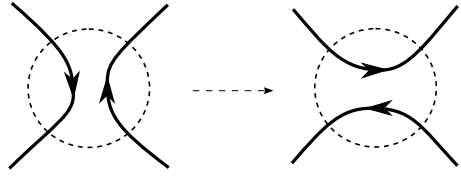}
 \end{center}
 \caption{saddle move}
 \label{saddle_move}
\end{figure}

\begin{figure}[h]
\begin{center}
  \includegraphics[width=100mm]{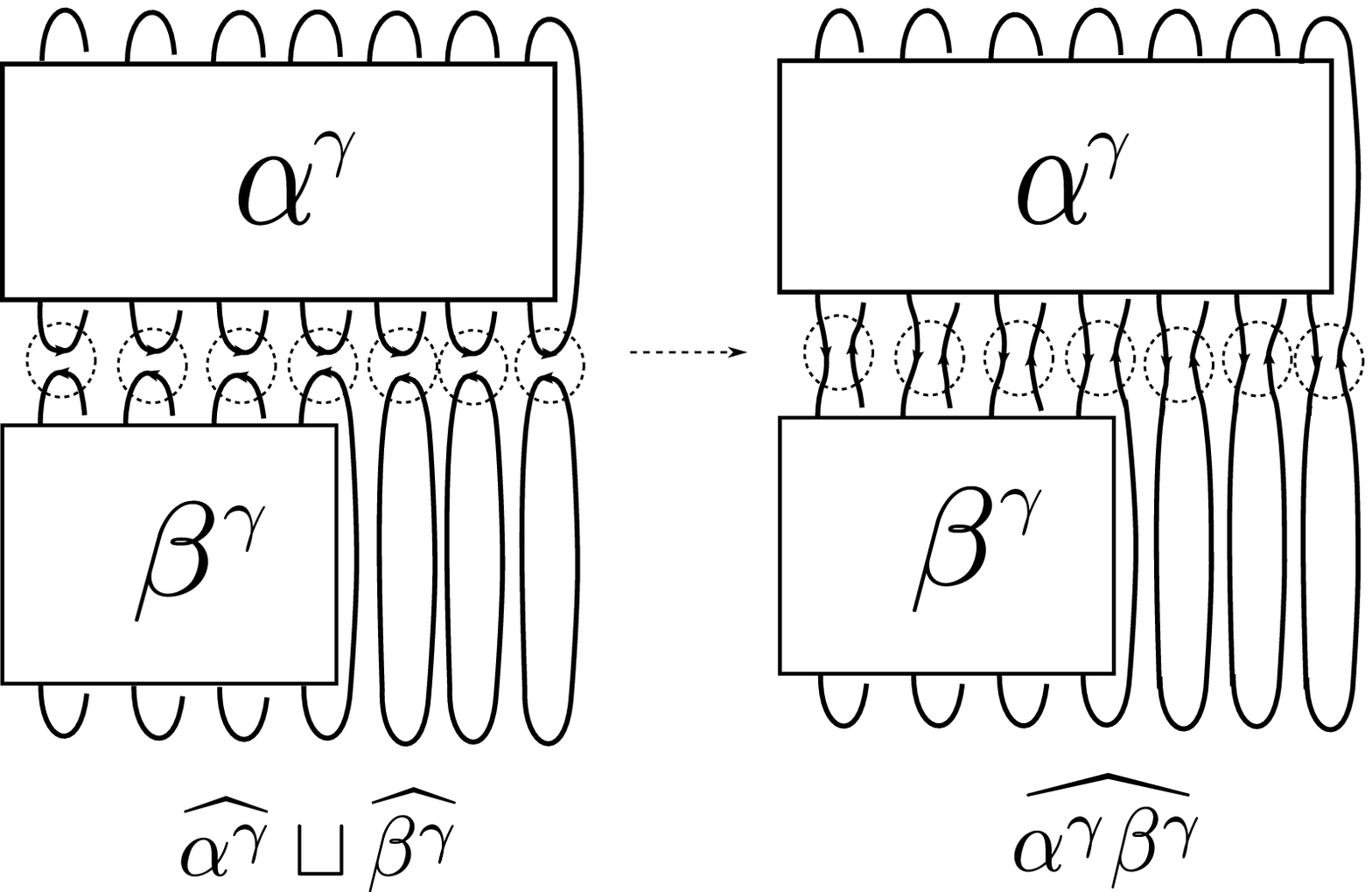}
 \end{center} 
 \caption{}
 \label{relative-quasi-morphism}
\end{figure}

Now we prove Theorem \ref{mainthm}.

\begin{lem}\label{welldef}
Let $G$ be a group normally generated by $K\subset G$ and $\nu$ a conjugation-invariant norm on $G$.
We assume that $\nu$ is bounded on $K$ (it is so when $K$ is finite). 
If $p,q\geq \sup_{k\in K}\nu(k)$, then $[G,G]_{\nu,p,q}=[G,G]$. 
\end{lem}

\begin{proof}
For $a,b,r,s \in G$, the following equalities hold:
\[ [ar,b] = [r,b]^a[a,b] , \quad [a,bs] = [a,b] [a,s]^b. \]
By using them, we can represent a element $f $ of $[G,G]$ as a product of commutators of the form $[k^*,l^*]$ ($k, l  \in K$). 
Therefore $f \in [G,G]_{\nu,p,q}$ if $p,q\geq \sup_{k\in K}\nu(k)$. 
\end{proof}

Since $B_\infty$ is normally generated by $\{\sigma_1^{\pm1}\}$, we have the following corollary.

\begin{cor}\label{well-def}
For $n\geq 2$ and $p,q\geq1$, $\cl_{\nu_n ,p,q}$ is well-defined on $[B_\infty,B_\infty]$.
\end{cor}

\noindent
\emph{(Proof of Theorem \ref{mainthm})} 
We apply Proposition \ref{kaw} to the signature $\sigma : B_\infty \to \mathbb{R}$. 
Since  $\sigma$ is a $\nu_n$-quasimorphism by Theorem \ref{rel}, 
it is sufficient to see that the stabilization of $\sigma$ is non-trivial and it is already known (see \cite{bk} for example). Therefore, for  $n\geq 2$ and $p,q\geq1$, $\cl_{\nu_n ,p,q}$ is a stably unbounded norm on $[B_\infty,B_\infty]_{\nu_n,p,q}=[B_\infty,B_\infty]$ by Corollary \ref{well-def}. \qed

\subsection{The extremal property}

We study the properties of the norms $\cl_{\nu_n,p,q}$. 
Two norms on a group are called \textit{equivalent} if their ratio is bounded away from 0 and $\infty$. First, we will show that the norms $\cl_{\nu_n,p,q}$ are equivalent each other. It follows from the fact that the norms have an ``extremal property''.

\begin{lem}\label{conj}
For $x,y,z \in G$, assume that $x$ and $y^z$ commute. 
Then $[x,y]$ is written as the products of 4 conjugates of $z$ or $z^{-1}$.
\end{lem}

\begin{proof}
By the assumption, $[x,[y,z]]=xy(y^{-1})^z x^{-1} y^z y^{-1}=xyx^{-1}y^{-1}=[x,y]$.  
Thus $[x, y]=[x,[y,z]]=x z^y z^{-1} x^{-1}z(z^{-1})^y=z^{xy}(z^{-1})^x z (z^{-1})^y$. 
\end{proof}

We call that a norm $\mu$ on $[B_\infty, B_\infty]$ \emph{has the extremal property} if $\mu$ satisfies the following condition: 
for any conjugation-invariant norm $\nu$ on $[B_\infty, B_\infty]$ which satisfies $\nu(\alpha^\beta)=\nu(\alpha)$ for all $\alpha \in [B_\infty,B_\infty]$ and $\beta \in B_\infty$ (not only $[B_\infty,B_\infty]$), there exists a positive number $\lambda>0$ such that $\nu \leq \lambda \mu $.

By using above lemma, we observe that the norms $ \cl_{\nu_n,p,q} $ have the property.

\begin{prop}\label{extrB}
For $n\geq 2$ and $p,q\geq1$, $\cl_{\nu_n,p,q} $ has the extremal property.
\end{prop}

\begin{proof}
First we prove that $\cl_{\nu_n,1,1}$ has the extremal property. 
Let $\alpha ,\beta \in B_\infty$ satisfy $\nu_n(\alpha)=\nu_n(\beta)=1$. 
Then there exists $\gamma,\gamma' \in B_\infty$ such that  $\widetilde{\alpha}:=\alpha^\gamma\in B_n$, 
$\widetilde{\beta}:=\beta^{\gamma'}\in B_n$. 
If $m$ is sufficiently large, $\widetilde{\alpha}^{\gamma' \gamma^{-1}}$ and 
$\widetilde{\beta}^{\Delta'_m} $ commute, where
\[\Delta'_m =  A'_{n,m-1}\cdots A'_{n,2} A'_{n,1} A_{n,2}^{\prime^{-1}} \cdots A^{\prime^{-1}}_{n,m-1}\]
and $A'_{n,i}$ is a \textit{commutator argyle braid} appeared in \cite{bk} .
By Lemma \ref{conj}, 
$[\alpha,\beta]^{\gamma'}=[\alpha^{\gamma'},\beta^{\gamma'}]
=[\widetilde{\alpha}^{\gamma' \gamma^{-1}} , \widetilde{\beta} ]$ 
is written as the products of 4 conjugates of $\Delta_{m}^{\prime^{\pm1}}$. 
Thus $\nu([\alpha,\beta])\leq 4\nu(\Delta'_m) $. 
Since $\Delta'_m$ is conjugate to $A'_{n,1}$, $\nu(\Delta'_m)=\nu(A'_{n,1})$.
Then we obtain $\nu \leq 4 \nu(A'_{n,1})  \cl_{\nu_n,1,1}$. 
Since $\cl_{\nu_n,1,1}\leq pq\cl_{\nu_n, p,q}$, the proposition follows.
\end{proof}

Since  $\cl_{\nu_n,p,q}(\alpha^\beta)=\cl_{\nu_n,p,q}(\alpha)$ for all $\alpha \in [B_\infty,B_\infty]$ and $\beta \in B_\infty$, it follows from Proposition \ref{extrB} that the norms $\cl_{\nu_n,p,q}$ are equivalent to each other. 
\medskip

Next, we also consider the property of the biinvariant word norm $\|\cdot\|:=q_{\{\sigma_1^{\pm 1}\}}$.

\begin{prop}\label{extrC}
The biinvariant word norm $\|\cdot\|$ has the extremal property.
\end{prop}

\begin{proof}
Let $\alpha \in [B_\infty,B_\infty]$ with $\| \alpha \|=k$. Then $\alpha$ is written as follows:
\begin{align*}
\alpha =& (\sigma_1^{\varepsilon_1})^{\alpha_1} (\sigma_1^{\varepsilon_2})^{\alpha_2}\cdots
 (\sigma_1^{\varepsilon_k})^{\alpha_k} \\
=& [\alpha_1, \sigma_1^{\varepsilon_1}] \sigma_1^{\varepsilon_1}
[\alpha_2, \sigma_1^{\varepsilon_2}] \sigma_1^{\varepsilon_2}
\cdots[\alpha_k, \sigma_1^{\varepsilon_k}] \sigma_1^{\varepsilon_k} \\
=& [\alpha_1, \sigma_1^{\varepsilon_1}] [\alpha_2, \sigma_1^{\varepsilon_2}]^{\sigma_1^{\varepsilon_1}} \sigma_1^{\varepsilon_1+\varepsilon_2}
\cdots[\alpha_k, \sigma_1^{\varepsilon_k}] \sigma_1^{\varepsilon_k} \\
=& \cdots \\
=& [\alpha_1, \sigma_1^{\varepsilon_1}] [\alpha_2, \sigma_1^{\varepsilon_2}]^{\ast}
\cdots[\alpha_k, \sigma_1^{\varepsilon_k}]^{\ast} \sigma_1^{\varepsilon_1+\cdots+\varepsilon_k},
\end{align*}
where $\varepsilon_i \in \{\pm1\}$ and $\alpha_i \in B_\infty$. 
Since $\alpha \in [B_\infty,B_\infty]$, $\varepsilon_1+\cdots+\varepsilon_k=0$. 
Thus $\alpha$ is written as a product of $k$ commutators of the form $[*, \sigma_1^{\pm1}]^*$. Since $\alpha_i$ and $(\sigma_1^{\varepsilon_i})^{\Delta'_m}$  commute in $B_\infty$ for a sufficient large $m$, $\nu([\alpha_i , \sigma_1^{\varepsilon_i}]^*)=\nu([\alpha_i , \sigma_1^{\varepsilon_i}])\leq 4\nu(\Delta'_m)=4\nu(A'_{2,1})$ and it follows that 
$\nu(\alpha)\leq 4\nu(A'_{2,1}) k=4\nu(A'_{2,1}) \|\alpha\|$.
\end{proof}

As a corollary of Proposition \ref{extrB} and \ref{extrC}, it follows that the norm $\|\cdot\|$ is also equivalent to the norms $\cl_{\nu_n,p,q}$ we constructed (Theorem \ref{equiv2}) .

\end{document}